\documentclass[12pt,reqno]{amsart}

\usepackage{amssymb}
\usepackage[pdftex,colorlinks,hypertexnames=false,linkcolor=blue,urlcolor=blue,citecolor=blue]{hyperref}
\usepackage{fullpage}

\newtheorem{theorem}{Theorem}[section]
\newtheorem{lemma}[theorem]{Lemma}
\newtheorem{proposition}[theorem]{Proposition}
\newtheorem{corollary}[theorem]{Corollary}

\theoremstyle{remark}
\newtheorem{remark}[theorem]{Remark}

\numberwithin{equation}{section}

\def\b{\mathfrak b}
\def\c{\mathfrak c}
\def\g{\mathfrak g}
\def\h{\mathfrak h}
\def\k{\mathfrak k}
\def\l{\mathfrak l}
\def\p{\mathfrak p}
\def\q{\mathfrak q}
\def\t{\mathfrak t}
\def\u{\mathfrak u}
\def\z{\mathfrak z}

\def\gl{\mathfrak{gl}}

\def\Z{\mathbb Z}

\def\cC{\mathcal C}
\def\cD{\mathcal D}
\def\cH{\mathcal H}

\def\cL{\mathcal L}
\def\cM{\mathcal M}
\def\cN{\mathcal N}
\def\cS{\mathcal S}
\def\cU{\mathcal U}

\def\cX{\mathcal X}
\def\cY{\mathcal Y}
\def\cCS{\mathcal{CS}}

\def\rA{\mathrm A}
\def\rB{\mathrm B}
\def\rC{\mathrm C}
\def\rD{\mathrm D}
\def\rE{\mathrm E}
\def\rF{\mathrm F}
\def\rG{\mathrm G}

\def\GL{\mathrm{GL}}

\def\bar{\overline}

\def\sub{\subseteq}
\def\Char{\operatorname{char}}

\def\Lie{\operatorname{Lie}}
\def\mod{\operatorname{mod}}
\def\rank{\operatorname{rank}}
\def\ssrank{\operatorname{ssrank}}

\def\reg{\mathrm{reg}}
\def\ss{\mathrm{ss}}
\def\ssm{\mathrm{ssm}}

\title{On commuting varieties of parabolic subalgebras}
\author{Russell Goddard and Simon M.~Goodwin}
\address{School of Mathematics,
University of Birmingham,
Birmingham, B15 2TT,
UK}
\email{rsg276@bham.ac.uk, s.m.goodwin@bham.ac.uk}

\begin{document}

\begin{abstract}
Let $G$ be a connected reductive algebraic group over an algebraically closed field $k$,
and assume that the characteristic of $k$ is zero or a pretty good prime for $G$.
Let $P$ be a parabolic
subgroup of $G$ and let $\p$ be the Lie algebra of $P$.   We consider the commuting variety
$\cC(\p) = \{(X,Y) \in \p \times \p \mid [X,Y] = 0\}$.
Our main theorem gives a necessary and sufficient condition for
irreducibility of $\cC(\p)$ in terms of the modality of the adjoint
action of $P$ on the nilpotent variety of $\p$.  As a consequence,
for the case $P = B$ a Borel subgroup of $G$,
we give a classification of when $\cC(\b)$ is irreducible;
this builds on a
partial classification given by Keeton.
Further, in cases where
$\cC(\p)$ is irreducible, we consider whether $\cC(\p)$ is a normal variety.
In particular, this leads to a classification of when $\cC(\b)$ is normal.
\end{abstract}

\maketitle

\section{Introduction}

Let $G$ be a connected reductive algebraic group over
an algebraically closed field $k$ with $\Char k = p \ge 0$, and
write $\g = \Lie G$ for the Lie algebra of $G$.
We assume that $p = 0$ or is a pretty
good prime for $G$; the notion of pretty good prime for $G$ was introduced by Herpel in
\cite[Definition 2.11]{He}, and is discussed more later in the
introduction.

In \cite{Ri}, Richardson proved that the commuting variety
$$
\cC(\g) = \{(X,Y) \in \g \times \g \mid [X,Y] = 0\}
$$
of $\g$ is irreducible for $\Char k = 0$. For $\g = \gl_n(k)$ this
was previously proved by Motzkin and Taussky in \cite{MT} and independently by Gerstenhaber in \cite{Ge}.
Richardson's result was subsequently extended
to positive (pretty good) characteristic by Levy in \cite{Le}.

There has been much recent research interest in $\cC(\g)$ and related varieties,
and the theory of commuting varieties finds applications in various areas of representation theory, and geometry.
We refer the reader for example to \cite{Bu}, \cite{Gi}, \cite{GS} \cite{Ng}, \cite{Pr} and \cite{PY}, and
the references therein.

Let $P$ be a parabolic subgroup of $G$ with Lie algebra $\p = \Lie P$.  It is a natural
generalization to consider the commuting variety
$$
\cC(\p) = \{(X,Y) \in \p \times \p \mid [X,Y] = 0\}
$$
of $\p$.
In case $P = B$ is a Borel subgroup and $p = 0$, the commuting
variety $\cC(\b)$ was considered in the PhD thesis of Keeton, \cite{Ke}.

Our main theorem gives a necessary and sufficient condition
for $\cC(\p)$ to be irreducible.
In the statement of the theorem we use the
modality of an algebraic group action, which is recalled in Section \ref{S:modality}.  Also we use the notation
$\cN(\k)$ for the variety of nilpotent elements in a subalgebra $\k$ of $\g$, and
we write $\rank G$ for the rank of $G$ and $\ssrank H$ for the semisimple rank of a Levi subgroup $H$ of $G$.

\begin{theorem} \label{T:main}
Let $G$ be a connected reductive algebraic group over an algebraically closed field $k$
with $\Char k = 0$ or a pretty good prime for $G$.  Let $P$ be a parabolic subgroup of $G$ and let $T$
be a maximal torus of $G$ contained in $P$.
Then the commuting variety $\cC(\p)$ is irreducible if and only if
$\mod(P\cap H : \cN(\p\cap \h))< \ssrank H$ for all Levi subgroups $H \ne T$ of $G$ containing $T$.
Moreover, if $\cC(\p)$ is irreducible, then $\dim \cC(\p) = \dim \p + \rank G$.
\end{theorem}

The assumption that $p =0$ or is a pretty good prime for $G$ is required in the statement of Theorem \ref{T:main}
is to ensure separability all orbit maps for the adjoint action of $P$ on $\p$, or in other words
that the scheme theoretic centralizer is smooth.  This does not
appear to be in the literature, and
is stated below and proved in Section \ref{S:parabolic}.

\begin{theorem} \label{T:Pseparableintro}
Let $G$ be a connected reductive algebraic group over an algebraically closed field $k$
with $\Char k = 0$ or a pretty good prime for $G$.  Let $P$ be a parabolic subgroup of $G$ and let $X \in \p$.
Then the adjoint orbit map $P \to P \cdot X$ is separable.
\end{theorem}

As mentioned above $\cC(\b)$ was investigated by Keeton in \cite{Ke}.  In particular,
\cite[Theorem 6.1]{Ke} includes an equivalent statement to Theorem \ref{T:main}
for the case $P = B$ and $p=0$.  Our methods build on those used in Keeton, but we require
a significantly different approach to deal with all parabolic subgroups, thus we
include all details.

Keeton proceeds to give a partial classification of irreducibility of $\cC(\b)$
in \cite[Sections 6.3 and 6.4]{Ke}.  In this case, we note that
$\cN(\b) = \u$ is the nilradical of $\b$.  We make use of recent results in \cite{GMR} and \cite{PS},
which allow us to determine $\mod(B:\u)$ for $G$ of sufficiently large rank in order
to give a complete classification of when $\cC(\b)$ is irreducible.

\begin{theorem} \label{T:borelclass}
Let $G$ be a connected reductive algebraic group over an algebraically closed field $k$
with $\Char k = 0$ or a pretty good prime for $G$.
Then $\cC(\b)$ is irreducible if and only if the type of each simple component of $G$ is one of the following.
\begin{itemize}
\item $\rA_l$ for $l \le 15$;
\item $\rB_l$ for $l \le 6$;
\item $\rC_l$ for $l \le 6$;
\item $\rD_l$ for $l \le 7$;
\item $\rG_2$ or $\rE_6$.
\end{itemize}
\end{theorem}

For $P \ne B$, less is known about the modalities $\mod(P:\cN(\p))$, and we only briefly discuss cases where we can determine whether
$\cC(\p)$ is irreducible or reducible at the end of this paper.  Here we just remark that from results of R\"ohrle in \cite{Ro1},
we can deduce that for a fixed value of $\ssrank L$ where $L$ is a Levi factor of $P$, we have that $\cC(\p)$ is reducible
if $\rank G$ is sufficiently large.
Also in very recent work for the case $G = \GL_n(k)$, Bulois--Boos in \cite[Main theorem]{BB} gives a classification of
cases where $\mod(P:\cN(\p)) = 0$, which gives cases where $\cC(\p)$ is irreducible.

For cases where $\cC(\p)$ is irreducible, we also consider the question of whether
$\cC(\p)$ is normal.  Our results in this direction are contained in
Section \ref{S:normality}.  In order to outline these results we recall
some background on commuting schemes.

The commuting scheme $\cCS(\g)$ of $\g$ is the subscheme of $\g \times \g$ defined by the
ideal $I_\g$ of $k[\g \times \g]$ generated by the regular functions
$(X,Y) \mapsto f([X,Y])$ for $f \in \g^*$; so $\cC(\g)$ is the underlying
variety of $\cCS(\g)$.
The question of whether $\cCS(\g)$ is reduced and normal
is a long-standing problem.    For $p = 0$, it was proved by Popov in \cite{Po} that
the singular locus of $\cCS(\g)$ has codimension 2, which reduces the problem to showing
that $\cCS(\g)$ is Cohen-Macaulay.  We also mention that Ginzburg
proved that the normalization of $\cC(\g)$ is Cohen-Macaulay in \cite{Gi}.

Under the assumption that $\cCS(\l)$
is Cohen--Macaulay for $L$ the Levi factor of $P$ containing $T$,
and that $\mod(P\cap H : \cN(\p\cap \h))< \ssrank H - 1$
for all Levi subgroups $H \ne T$ of $G$ containing $T$ with $\ssrank H > 1$,
we prove normality of $\cC(\p)$ in Theorem \ref{T:normality}.
Our methods build on those of Keeton in \cite[Section 6.2]{Ke}, and we
show that the singular locus of the commuting scheme of $\p$ has codimension 2.
We note that this statement about the singular
locus holds without the Cohen--Macaulay assumption on $\cCS(\l)$.

For the case $P = B$,
the commuting scheme of $\l = \t$ is certainly Cohen--Macaulay,
and we show that the modality condition is also sufficient
in Theorem \ref{T:borelnormality}.
This leads a classification of when $\cC(\b)$
is normal in Theorem \ref{T:borelnormalityclass}.

We mention that another motivation for investigating $\cC(\p)$ is that one can hope to understand
$\cC(\g)$ through the fibre bundle $G \times^P \cC(\p)$,
where $P$ acts on $\cC(\p)$ diagonally by the adjoint action.  To recall the definition
of  $G \times^P \cC(\p)$, we consider the action of $P$ on $G \times \cC(\p)$ by
$x \cdot (g,(X,Y)) = (gx^{-1},(x \cdot X,x \cdot Y))$ for $x \in P$, $g \in G$, and $X,Y \in \p$.
Then the set of $P$-orbits in $G \times \cC(\p)$ has the structure of a variety, this can
be explained in direct analogy with the construction given in \cite[\S8.11]{Ja}.
Further, there is a surjective morphism $G \times^P \cC(\p) \to \cC(\g)$ given by
$[(g,(X,Y))] \mapsto (g \cdot X, g \cdot Y)$.  This construction is considered for the case
$P = B$ in \cite[\S5.6]{Ke}.

We end the introduction by mentioning some related recent results, and then
giving some remarks about our main results.

In \cite{BE}, Bulois and Evain investigated the irreducibility and equidimensionality of $\cC(\cN(\p))$
for certain parabolic subgroups of $\g = \gl_n(k)$.
In \cite{GR}, R\"ohrle and the second author classified when $\cC(\u)$ is irreducible,
where $\u = \cN(\b)$ is the nilradical of $\b$, and also determined the irreducible
components of $\cC(\u)$ in minimal cases where it is not irreducible.

We note that Theorem \ref{T:main} can be reduced to simple $G$, though
this is not required for our proof.  To see this
we write $G$ as a central product $G = Z(G) G_1\cdots G_m$, where $G_1,\dots,G_m$
are the simple components of $G$, and $P_i = P \cap G_i$.  Then we have
$\p = \z(\g) \oplus \g_1 \oplus \dots \oplus \g_m$, and it is straightforward
to prove that $\cC(\p) = \cC(\z(\g)) \times \cC(\p_1) \times \cdots \times \cC(\p_m)$.
Thus $\cC(\p)$ is irreducible if and only if $\cC(\p_i)$ is irreducible
for each $i$.  Further, we can easily see that $\mod(P \cap H:\cN(\p \cap \h)) =
\mod(P_1 \cap H :\cN(\p_1 \cap \h))+ \dots + \mod(P_m \cap H:\cN(\p_m \cap \h))$ for $H$
a Levi subgroup of $G$ containing $T$.
This implies that the modality condition on $P$ in Theorem \ref{T:main}
holds if and only if it holds for each $P_i$.

For $G$ simple it seems plausible that the condition
$\mod(P\cap H : \cN(\p\cap \h))< \ssrank H$ for all Levi subgroups $H \ne T$ of $G$ containing $T$
can be replaced by the single condition that $\mod(P : \cN(\p))< \ssrank G$ in Theorem \ref{T:main}.
As consequence of Theorem \ref{T:borelclass}, we can, a fortiori,
weaken our condition in this way for $P = B$.
The methods used in our proof do not get around the need for the inductive assumption.
However, we do not consider this to be a serious limitation, because if
$\mod(P:\cN(\p))$ has been determined, then most likely $\mod(P\cap H : \cN(\p\cap \h))$
can be determined for all $H$ as in the statement as well.

As already mentioned the assumption
that $p =0$ or is a pretty good prime for $G$ is essentially required
for Theorem \ref{T:Pseparableintro}.  We recall from \cite[Lemma 2.12]{He},
that $p$ is a pretty good prime for $G$ if and only if $p$ is a good prime for $G$
and there is no $p$-torsion in both the quotient of the character group of $G$
by the root lattice and the quotient of the cocharacter group of $G$ by the
coroot lattice.
Also we note that $p$ is pretty good is implied by the
{\em standard hypothesis} from \cite[\S2.9]{Ja}:
\begin{itemize}
\item[(H1)] The derived subgroup of $G$ is simply connected;
\item[(H2)] $\Char k$ is zero or a good prime for $G$; and
\item[(H3)] there is a nondegenerate $G$-invariant symmetric bilinear form on $\g$.
\end{itemize}
In fact as shown in \cite[Theorem 5.2]{He}, we have that $p > 0$ is pretty good for $G$ if and only if
we can obtain $G$ from a reductive group $G'$ satisfying (H1)--(H3) after a finite number of the
operations:
\begin{itemize}
\item[(R1)] the replacement of $G'$ by a separably isogenous group $G''$; or
\item[(R2)] the replacement of $G' = G'' \times S$ by $G''$, where $S$ is a torus.
\end{itemize}
(We recall that a separable isogeny $\pi : G' \to G''$ is a surjective
homomorphism with finite kernel such that $d\pi$ is an isomorphism, and
we say that $G'$ and $G''$ are separably isogenous if there is a separable
isogeny between them.)

We remark that in \cite{Le}, Levy proved that $\cC(\g)$ is irreducible under the
hypothesis (H1)--(H3).  Using \cite[Theorem 5.2]{He}, we
can see that irreducibility of $\cC(\g)$ holds when $p =0$ or is a pretty good prime.
Indeed if $p =0$ or is pretty good prime for $G$, then there exists
$G'$ satisfying (H1)--(H3) such that $\g \cong \g'$.

Our methods require the assumption that $p$ is pretty good.  In particular, this
is required for Theorem \ref{T:Pseparableintro}.  In low rank examples, we can already see that
the situation is different when $p$ is not very good.  For example,
for $p=2$, and $G = \mathrm{SL}_2(k)$, we have that $\b$ is abelian, so that $\dim \cC(\b) > \dim B + \rank G$.
Also for $p=2$ and $G = SO_5(k)$, it is possible to show that $\cC(\g)$ is not irreducible.
It would be interesting to understand the general situation for $p$ not pretty
good.

We note that our methods apply also to the commuting variety
$$
\cC(P) = \{(x,y) \in P \times P \mid xy = yx\}
$$
of $P$.  Thus with some modifications, analogous results about
$\cC(P)$ can be proved.

\medskip

\noindent
{\bf Acknowledgments:}  This research forms part of the first author's PhD research and he is
grateful to the EPSRC for financial support.  We also thank Michael Bulois and Paul Levy
for helpful discussions related to this research.

\section{Algebraic group actions, modality and commuting varieties} \label{S:modality}
Let $G$ be a linear algebraic group over an algebraically closed field $k$ with Lie algebra $\g$,
and let $V$ be a variety on which $G$ acts morphically, in the sense of \cite[\S1.7]{Bo}.
Let $g \in G$, $v \in V$ and let $U$ be a subvariety of $V$.
We write $g \cdot v$ for the image of $v$ under $g$, and $G \cdot v = \{g \cdot v \mid g \in G\}$ for the $G$-orbit of $v$.
The stabilizer of $v$ in $G$ is denoted $C_G(v) = \{g \in G \mid g \cdot v = v\}$.
We write  $g \cdot U = \{g \cdot u \mid u \in U\}$ for the image of $U$ under $g$,
and $G \cdot U = \{g \cdot u \mid g \in G, u \in U\}$ for the
$G$-saturation of $U$.
The normalizer of $U$ in $G$ is denoted $N_G(U) = \{g \in G \mid g \cdot u \in U \text{ for all } u \in U\}$.

Below we recall the definition of the modality of $G$ on $V$, and the sheets of $G$ on $V$.
In order to do this we set $V_j=\{ v\in V \mid \dim G \cdot v=j\}$ for $j \in \Z_{\ge 0}$.

The {\em modality of $G$ on $V$} is defined to be
$$
\mod(G:V)= \max_{j \in \Z_{\ge 0}} (\dim V_j - j).
$$
Informally, $\mod(G:V)$ is the maximum number of parameters on which a family of $G$-orbits depends.
The notion of modality originates in the work of Arnold \cite{Ar}, we refer the reader also to \cite{Vi} and \cite[Section 5.2]{PV}.

We define $\cS(G,V)_j$ to be the set of irreducible components of $V_j$, and let
$$
\cS(G,V) = \bigsqcup_{j \in \Z_{\ge 0}} \cS(G,V)_j.
$$
The elements of $\cS(G,V)$ are called {\em sheets of $G$ on $V$}.  The
notion of sheets was introduced by Borho and Kraft in \cite{BK}.  We remark that in the case
that $G$ acts on $V$ with finitely many orbits, the sheets of $G$ on $V$ coincide
with the orbits of $G$ on $V$.

Next we recall that the commuting variety of $\g$ is
$$
\cC(\g)=\{(X,Y) \in \g \times \g  \mid [X,Y]=0\}.
$$
The connection between commuting varieties and modalities is made clear in the Lemma \ref{L:dimcg}.
Before stating this lemma we briefly discuss separability of the orbit maps for the adjoint
action of $G$ on $\g$.  Let $X \in \g$, write $\phi_X : G \to G \cdot X$ for the orbit map for $X$
and $\c_\g(X) = \{Y \in \g \mid [Y,X] = 0\}$ for the centralizer of $X$ in $\g$.
We recall from \cite[Proposition 6.7]{Bo} that separability of $\phi_X$ is equivalent to $\Lie C_G(X) = \c_\g(X)$ (or in other
words that the scheme theoretic centralizer is smooth).
Therefore, given the assumption that orbit maps are separable, Lemma \ref{L:dimcg}
can be proved with the same argument as in characteristic zero, see for example \cite[Lemma 2.1]{GR}.

\begin{lemma}
\label{L:dimcg}
Let $G$ be a linear algebraic group with Lie algebra $\g$.
Assume that for all $X \in \g$, the orbit map $G \to G \cdot X$ is separable.
Then
$$
\dim \cC(\g) = \dim \g + \mod(G:\g).
$$
\end{lemma}

\section{Notation for reductive groups} \label{S:notation}

We introduce the notation used in the remainder of this paper, and recall some standard
results about reductive algebraic groups.

Let $G$ be a reductive algebraic group over an algebraically closed field $k$ of characteristic $p \ge 0$
and let $\g = \Lie G$ be the Lie algebra of $G$.  We assume throughout that $p = 0$ or $p$ is a pretty good prime for $G$.

For a closed subgroup $K$ of $G$, we write $\k = \Lie K$ for the Lie algebra of $K$.
For a reductive subgroup $H$ of $G$, we write $\rank H$ for the rank
of $H$ and $\ssrank H$ for the semisimple rank of $H$.
For a subalgebra $\k$ of $\g$, we denote the centre of $\k$ by $\z(\k)$, and
the variety of nilpotent elements of $\k$
by $\cN(\k)$.

Let $P$ be a parabolic subgroup of $G$ and let
$T \sub P$ be a maximal torus of $G$.
Let $L$ be the Levi factor of $P$
containing $T$ and let $U_P$ be the unipotent radical of $P$; the Lie algebra of $U_P$ is denoted
$\u_P$.  Let $B$ be a Borel subgroup of $G$ contained in $P$ and containing $T$,
and let $U = U_B$ be the unipotent radical of $U$.

We write $\Phi$ for the root system of $G$ with respect to $T$.  For $\alpha \in \Phi$,
we denote the corresponding root subspace of $\g$ by $\g_\alpha$ and let $E_\alpha$ be a generator for $\g_\alpha$.
For a subalgebra $\k$ of $\g$ stable under the adjoint action of $T$, we write
$\Phi(\k)$ for the set of $\alpha \in \Phi$ such that $\g_\alpha \sub \k$.
Let $\Phi^+ = \Phi(\b)$ be
system of positive roots determined by $B$ and let $\Pi \sub \Phi$
be the corresponding set of simple roots.

We recall that a  subgroup $H$ of $G$ is called a Levi subgroup
if it is a Levi factor of some parabolic subgroup of $G$.  This
is equivalent to $H = C_G(S)$ where $S$ is a torus in $G$,
or, under the assumption that $p$ is good, $H = C_G(X)$ where
$X \in \g$ is semisimple.

For $p >0$, we note that, under our assumption that $p$ is a pretty good
prime for $G$, we also have that $p$ is a pretty good prime for any
Levi subgroup of $G$.  This follows easily from \cite[Definition 2.11]{He}.

Given a subset $J$ of $\Pi$, we let $\Phi_J$
be the closed subsystem of $\Phi$ generated by $J$.
We say that a Levi subgroup $H$ of $G$ is a standard Levi subgroup if
$H$ contains $T$ and $\Phi(\h) = \Phi_J$ for some $J \sub \Pi$, or equivalently
if $H$ is the Levi factor containing $T$ of a parabolic subgroup containing $B$.
In particular, $L$ is a standard Levi subgroup of $G$ and we let $I$ be the subset of $\Pi$
such that $\Phi(\l) = \Phi_I$.

We recall that $X \in \g$ is called {\em regular} if $\dim \c_\g(X) = \rank G$.
We write $\g^\reg$ for the set of regular elements in $\g$, and for a subalgebra $\k$ of $\g$,
we let $\k^\reg = \k \cap \g^\reg$.

The set of regular semisimple elements in $\g$ is denoted by $\g^\reg_\ss$.
Next we observe that under our assumption that $p=0$ or is a pretty good prime for $G$, there exist
regular semisimple elements in $\t$; we expect this is well-known.

We view $\Phi \sub X(T)$, where $X(T)$ is the character group of $T$, and given $\alpha \in \Phi$, and write $d\alpha : \t \to k$ for its
differential.  To check there are regular semisimple elements in $\t$, we need to observe that there
exists $X \in \t$ such that $d\alpha(X) \ne 0$ for all
$\alpha \in \Phi$.  It suffices to check that for any $\alpha \in \Phi$, we have that
$d\alpha$ is nonzero.  Let $\alpha \in \Phi$.  We may write $\alpha = m\beta$ where $\beta$ is an indivisible
element of $X(T)$ and $m \in \Z_{\ge 1}$.  Then there is $m$-torsion in $X(T)/\Z\Phi$, so that $p \nmid m$, because
$p$ is pretty good for $G$.  There exists a cocharacter $\lambda : k^\times \to T$ such that $\langle \beta,\lambda \rangle = 1$,
and we write $X = d\lambda(1) \in \t$.  Then we have $d\alpha(X) = m$, where we view $m$ as an element of $k$, which is
nonzero in $k$.

Now using \cite[Theorem 1]{Le}, the proof of \cite[Theorem 2.5]{Hu} can be adapted to prove that the regular semisimple
elements form an open subvariety of $\g$.  In turn it follows that $\g^\reg$ is an open subvariety,
and that the minimal dimension of a centralizer in $\g$ is $\rank G$.

\section{Separability of orbit maps for the adjoint action of a parabolic subgroups} \label{S:parabolic}

The goal of this section is to prove Theorem \ref{T:Pseparableintro}.
The main idea of the proof is based on arguments
in \cite[Chapter I \S5]{SS}, see also \cite[Theorem 2.5]{Ja}, where
the theorem is proved in the case $P = G$.
Also the ideas for the proof of \cite[Theorem 3.3]{He} are used for
Lemma \ref{L:sepreduction}, which
gives a reduction that we use in the proof of Theorem \ref{T:Pseparableintro}.

We require some preliminary discussion for the statement of
Lemma \ref{L:sepreduction}.  First we recall that the operations (R1)--(R2) are
stated in the introduction.

Suppose
$\pi : G \to G'$ is a separable isogeny, where $G'$ be a reductive
algebraic group, i.e.\ $G$ is obtained from $G'$ by an (R1) operation.  Then $P' = \pi(P)$ is a parabolic subgroup
of $G$ and the proof of \cite[Proposition 2.8]{Ja} can be adapted to show that
the adjoint orbit map $P \to P \cdot X$ is separable
for all $X \in \p$ if and only if the adjoint orbit map $P' \to P' \cdot X'$ is separable
for all $X' \in \p'$.

Now suppose that $G' = G \times S$, where $S$ is a torus, i.e.\ $G$ is obtained from $G'$ by an (R2) operation.  Then $P' = P \times S$
is a parabolic subgroup of $G'$.  Further, since $S$ is contained in the centre of $G$,
it is clear that the adjoint orbit map $P \to P \cdot X$ is separable
for all $X \in \p$ if and only if the adjoint orbit map $P' \to P' \cdot X'$ is separable
for all $X' \in \p'$.

Iteration of the arguments in the previous two paragraphs proves
the following lemma.

\begin{lemma} \label{L:sepreduction}
Suppose that $G$ is obtained from the reductive algebraic group $G'$ by
a finite number of the operations {\em (R1)--(R2)}, and let $P'$ be the parabolic subgroup
of $G'$ corresponding to $P$.  Then the adjoint orbit map $P \to P \cdot X$ is separable
for all $X \in \p$ if and only if the adjoint orbit map $P' \to P' \cdot X'$ is separable
for all $X' \in \p'$.
\end{lemma}

Armed with Lemma \ref{L:sepreduction}, we are ready to proceed with the proof of Theorem \ref{T:Pseparableintro}.

\begin{proof}[Proof of Theorem \ref{T:Pseparableintro}]
By Lemma \ref{L:sepreduction} and the discussion at the end of the introduction,
we may assume that $G$ satisfies (H1)--(H3).
It follows from the arguments in \cite[\S2.6--2.9]{Ja}
that if $p$ is pretty good for $G$, then
$G$ can be obtained after a finite number of the
operations (R1)--(R2) from a reductive group $G'$ with a
faithful representation $\rho : G' \to \GL_n(k)$
such that restriction of the trace form on $\gl_n(k)$ to
$d\rho(\g')$ is nondegenerate.
Hence by Lemma \ref{L:sepreduction}, we may assume
that there is a faithful representation $\rho : G \to \GL_n(k)$
such that $d\rho : \g \to \gl_n(k)$ is injective and the restriction of the trace form on $\gl_n(k)$ to
$d\rho(\g)$ is nondegenerate.
We use $\rho$ to identify $G$ with a subgroup of $\bar G = \GL_n(k)$, and $d\rho$ to identify
$\g$ with a subalgebra of $\bar \g = \gl_n(k)$.

Recall that $I = \{\alpha \in \Pi \mid \g_\alpha \sub \l\}$ and let $m$ be the index
of the cocharacter group of $T$ in the group of coweights of $T$.
Let $\lambda : k^\times \to T$ be the unique cocharacter with image inside the derived subgroup of $G$ satisfying
$$
\langle \lambda, \alpha \rangle =
\begin{cases}
0 &\text{if } \alpha \in I  \\
m & \text{if } \alpha \in \Pi \setminus I,
\end{cases}
$$
where $\langle \cdot, \cdot \rangle$ is the pairing between cocharacters and characters of $T$.

Then $\lambda$ determines a grading
$$
\g = \bigoplus_{j \in \Z} \g(\lambda;j)
$$
where $\g(\lambda;j) = \{Y \in \g \mid \lambda(t) \cdot Y = t^j Y\}$.  We
have
$$
\p = \bigoplus_{j \in \Z_{\ge 0}} \g(\lambda;j) \quad \text{and} \quad \l = \g(\lambda;0).
$$
Similarly $\lambda$ defines a grading
$$
\bar \g = \bigoplus_{j \in \Z} \bar \g(\lambda;j)
$$
and we define
$$
\bar \p = \bigoplus_{j \in \Z_{\ge 0}} \bar \g(\lambda;j).
$$

Since the restriction of the trace form on $\bar \g$ to $\g$ is nondegenerate,
we may decompose $\bar \g = \g \oplus \tilde \g$ as a $G$-module, where $\tilde \g$ is the
orthogonal complement to $\tilde \g$ in $\bar \g$.  We let $\tilde \p = \tilde \g \cap \bar \p$.
Then we see that $\bar \p = \p \oplus \tilde \p$ as $P$-modules.

We write $\bar P$ for the parabolic subgroup of $\bar G$ with Lie algebra $\bar \p$.
Then we have that $C_{\bar P}(X)$ consists of the invertible elements of $\c_{\bar \p}(X)$.
Therefore, $\dim C_{\bar P}(X) = \dim \c_{\bar \p}(X)$, and thus $\Lie C_{\bar P}(X) = \c_{\bar \p}(X)$.

Let $f : \bar P \to \bar \p$ be the map defined by $f(g) = g \cdot X - X$.
The differential $(df)_1 : \bar \p \to \bar \p$ of $f$ at $1$ is given by $(df)_1(Y) = [Y,X]$.
We have
$$
\dim f(\bar P) = \dim \bar P \cdot X = \dim \bar P - \dim C_{\bar P}(X) = \dim \bar \p - \dim \c_{\bar \p}(X) = \dim([\bar \p,X]),
$$
and it follows that $T_0(f(\bar P)) = [\bar \p,X]$.

We have
\begin{align*}
T_0(f(P))
 & \sub T_0(f(\bar P)) \cap T_0(\p) \\
 & = [\bar \p,X] \cap \p \\
 & = [\p,X]\\
 &  \sub T_0(f(P)).
\end{align*}
The first inclusion here is immediate, whilst the first equality
follows from $T_0(f(\bar P)) = [\bar \p,X]$.  To see the second equality holds,
consider $\bar Y = Y + \tilde Y \in \bar \p$
with $Y \in \p$ and $\tilde Y \in \tilde \p$.
Then
$$
[\bar Y,X] = [Y + \tilde Y,X]
= [Y,X] + [\tilde Y,X]
$$
where $[Y,X] \in \p$ and $[\tilde Y,X] \in \tilde \p$ as
$\bar \p = \p \oplus \tilde \p$ is a $P$-module decomposition.
We deduce that $[\bar \p,X] \cap \p = [\p,X]$.  The last inclusion
holds because $[\p,X] = (df)_1(\p)$.

Hence, we have that $[\p,X] = T_0(f(P))$.
Therefore,
$$
\dim C_P(X)  = \dim P - \dim(P \cdot X) = \dim P -\dim f(P) =
\dim \p - \dim([\p,X]) = \dim \c_\p(X),
$$
which implies that $\Lie C_P(X) = \c_\p(X)$ and thus
that the orbit map $P \to P \cdot X$ is separable.
\end{proof}

\begin{remark}
The key ingredient in our proof of Theorem \ref{T:Pseparableintro}
is the existence of $\bar P \le \bar G$ such that
$P = \bar P \cap G$,
the adjoint orbit maps $\bar P \to \bar P \cdot X$ are separable, and
there is a $P$-module decomposition $\bar \p = \p \oplus \tilde \p$.
Thus the proof is applicable in other situations, for example we could
replace $P$ by any subgroup of $K$ of $G$ such that $K$ is normalized by $T$,
the restriction of the trace form to $d\rho(\t \cap \k)$ is nondegenerate,
and $\Phi(\k)$ is closed under addition within the set of
weights of $T$ in $\bar \g$.
\end{remark}

\section{On irreducibility of \texorpdfstring{$\cC(\p)$}{C(p)}}

The goal of this section is to prove Theorem \ref{T:main}.
Before getting to this proof, we give some preliminary results.
For the first lemma,
we define $\cC_\ss^\reg(\p)$ to be the subvariety of $\cC(\p)$ consisting of
pairs of commuting regular semisimple elements.

\begin{lemma}
\label{L:irrcomp}
$ $
\begin{itemize}
\item[(a)] $\cC_\ss^\reg(\p)$ is open
and irreducible in $\mathcal{C}(\mathfrak{p})$,
\item[(b)] the closure $\bar{\cC_\ss^\reg(\p)}$ of $\cC_\ss^\reg(\p)$
is an irreducible component of $\cC(\p)$, and
\item[(c)]  $\dim \cC_\ss^\reg(\p)=\dim \p+\rank G$.
\end{itemize}
\end{lemma}

\begin{proof}
  Then we can deduce that the variety of regular semisimple elements $\p_\ss^\reg = \p \cap \g_\ss^\reg$ in $\p$ is open in
$\p$.  It follows that $\p_\ss^\reg\times \p_\ss^\reg$ is open in $\p \times \p$.
Thus $\cC_\ss^\reg(\p) = \cC(\p) \cap( \p_\ss^\reg \times \p_\ss^\reg)$
is open in $\cC(\p)$.

Consider the map $\mu: P \times (\t^\reg \times \t^\reg) \to \cC_\ss^\reg(\p)$ defined by $\mu(g,X,Y)=(g \cdot X, g \cdot Y)$;
we recall that $\t^\reg$ is the set of regular elements in $\t$.
Let $(X,Y) \in \cC_\ss^\reg(\p)$.  There exists a maximal torus $T_1$ of $P$ such that $X \in \t_1$. We have
$\c_\p(X) = \t_1$, as $X$ is regular semisimple, so that $Y \in \t_1$ too.
Since maximal tori of $P$ are conjugate, it follows that $\mu$ is surjective.
Therefore, $\cC_\ss^\reg(\p)$ is irreducible.

Having proved $\cC_\ss^\reg(\p)$ is both open and irreducible, we deduce
that $\bar{\cC_\ss^\reg(\p)}$ is an irreducible component of $\cC(\p)$.

Now suppose that $\mu(g,X,Y)=\mu(1,X',Y')$ for some $g \in P$ and $X,Y,X',Y' \in \t^\reg$. Then $g \cdot X = X'$, so
$g \cdot \t = g \cdot \c_\p(X) = \c_\p(X') = \t$.
Hence, $g \in N_P(\t)$.
It follows that the dimension of each fibre of $\mu$ is
$\dim N_P(\t)= \dim \t = \rank G$.
Therefore,
\begin{align*}
\dim \cC_\ss^\reg(\p) = \dim P + \dim \t^\reg+ \dim \t^\reg -\dim \t =\dim \p + \rank G. & \qedhere
\end{align*}
\end{proof}

We have the following immediate corollary.

\begin{corollary}
$\cC(\p)$ is irreducible if and only if $\cC(\p)=\bar{\cC_\ss^\reg(\p)}$.
\end{corollary}

We note that Richardson proved that $\cC(\g)$ is irreducible by proving that $\cC(\g)=\bar{\cC_\ss^\reg(\g)}$  in \cite{Ri}; this is
 also the case for Levy's proof in \cite{Le} for positive characteristic.  In particular, we have
 $\dim \cC(\g) = \dim G + \rank G$.

Our next lemma gives a lower bound for the dimensions of irreducible components of $\cC(\p)$.

\begin{lemma} \label{L:compdim}
All irreducible components of $\cC(\p)$ have dimension at least $\dim \p + \rank G$.
\end{lemma}

\begin{proof}
Consider the Levi decomposition $\p =\l \oplus \u_P$.
Let $X,Y \in \p$, and write $X=X_1+X_2$, $Y=Y_1+Y_2$,
where $X_1,Y_1 \in \l$ and $X_2,Y_2 \in \u_P$. Then $[X,Y]=[X_1,Y_1]+[X_1,Y_2]+[X_2,Y_1]+[X_2,Y_2]$;
we note that $[X_1,Y_1] \in \l$ and the remaining terms are in $\u_P$.

Let $\cL$ be the subvariety of $\p \times \p$ of pairs $(X,Y)$ such that $(X_1,Y_1) \in \cC(\l)$ and $X_2,Y_2 \in \u_P$.
Then we have that $\cC(\p) \sub \cL$ and $\cL \cong \cC(\l) \times (\u_P \times \u_P)$.
Since $\cC(\l)$ is irreducible of dimension
$\dim \l + \rank L$, we have that $\cL$ is irreducible of dimension
$\dim \l + \rank G + 2\dim \u_P = \dim \p + \rank G + \dim \u_P$.

Consider the commutator map $F : \cL \to \u_P$ given by $F(X,Y) = [X,Y]$.
Then $\cC(\p)$ is the zero fibre of $F$,
so the codimension in $\cL$ of each irreducible component of $\cC(\p)$ is at most $\dim \u_P$.
Therefore, each irreducible component of $\cC(\p)$ has dimension at least
\begin{align*}
(\dim \p + \rank G+ \dim \u_P)-\dim \u_P = \dim \p + \rank G. & \qedhere
\end{align*}
\end{proof}

We are now ready to prove our main theorem.

\begin{proof}[Proof of Theorem \ref{T:main}]
In order to prove this theorem we decompose $\cC(\p)$ as a disjoint union of
irreducible subvarieties such that the closure of some of these subvarieties
are the irreducible components of $\cC(\p)$. This allows us to determine when $\cC(\p)$
is irreducible and is achieved by partitioning
the $P$-orbits in $\p$ in a way that generalizes the
partition of $G$ orbits in $\g$ into decomposition classes as given by Borho and
Kraft in \cite{BK}.  This is similar to the approach used by Popov
when considering $\cC(\g)$ in \cite{Po}.

Let $\hat \cH$ denote the set of all Levi subgroups $H$ of $G$ containing $T$.
We note that $H \in \hat \cH$ is determined by the set of $\alpha \in \Phi$
such that $\g_\alpha \sub \h$.
Thus we see that $\hat \cH$ is a finite set.
We note that different Levi subgroups in $\hat \cH$ may be conjugate under $P$, and we choose
$\cH$ to be a subset of $\hat \cH$ containing one representative from
each $P$-conjugacy class.

Let $H \in \cH$.
We define
$\z(\h)^\reg = \{X \in \z(\h) \mid \c_\g(X) = \h\}$. We have
that $\z(\h)^\reg$ is open in $\z(\h)$, so $\z(\h)^\reg$ is irreducible and $\dim \z(\h)^\reg = \dim \z(\h) =  \rank G - \ssrank H$.

We have that $P \cap H$ is a parabolic subgroup of $H$, and we also consider $N_P(P \cap H)$.
We note that $P \cap H$ has finite index in $N_P(P \cap H)$, because any coset
in $N_P(P \cap H)/(P \cap H)$ can be chosen to have a representative that normalizes $T$.  Thus $\dim N_P(P \cap H)
= \dim(P \cap H)$.

We write $\cS_H = \cS(P \cap H,\p \cap \h)$
for the set of sheets of $P \cap H$ on $\cN(\p \cap \h)$
and $\cS_{H,j} = \cS(P \cap H,\cN(\p \cap \h))_j$.
So we have $\cS_H = \bigsqcup_{j \in \Z_{\ge 0}} \cS_{H,j}$.
We note that $N_P(P \cap H)$ acts of $\cN(\p \cap \h)$, and this gives rise
to an action of $N_P(P \cap H)/(P \cap H)$ on $\cS_H$.

Let $X \in \p$ with Jordan decomposition $X = X_s + X_n$.  Up to
the adjoint action of $P$ we may assume that $X_s \in \z(\h)^\reg$
for some $H \in \cH$.  Then we have that $X_n \in \cN(\p \cap \h)$, so
$X_n \in S$ for some $S \in \cS_H$.

Let $S \in \cS_{H,j}$, where $j \in \Z_{\ge 0}$.  We define
$\cM_{H,S} \sub \p$ to be the variety of all $X = X_s + X_n \in \p$
such that $X_s \in \z(\h)^\reg$ and $X_n \in S$.  We have $\cM_{H,S} \cong \z(\h)^\reg \times S$, so $\cM_{H,S}$ is irreducible
and $\dim \cM_{H,S} = \rank G - \ssrank H + \dim S$.

Let $\cC'_{H,S} = \{(X,Y) \mid X \in \cM_{H,S}, Y \in \c_\p(X)\}$.
For $X  \in \cM_{H,S}$ with Jordan decomposition $X = X_s + X_n$,
we have $\c_\p(X) = \c_\p(X_s) \cap \c_\p(X_n) = \c_{\p \cap \h}(X_n)$.
Thus $\dim \c_\p(X) = \dim \c_{\p \cap \h}(X_n)
= \dim(\p \cap \h) - j$, where for the last equality we
require Theorem \ref{T:Pseparableintro}.
Therefore, the dimension of $\c_\p(X)$ does not depend
on the choice of $X \in \cM_{H,S}$.

Let $\cX$ be an irreducible component of
$\cC'_{H,S}$ and consider the morphism
$\pi : \cX \to \cM_{H,S}$ given by projecting on to the first component.
The function taking $(X,Y) \in \cX$ to the maximal dimension
of an irreducible component of $\pi^{-1}(X,Y)$ containing $(X,Y)$
is upper semi-continuous, see for example \cite[\S8 Corollary 3]{Mum}.
Thus the set of $X \in \cM_{H,S}$ such that
$\{X\} \times \c_\p(X) \sub \cX$ is closed in $\cM_{H,S}$; here we
require that $\dim \c_\p(X)$ does not depend on $X \in \cM_{H,S}$ as
proved above.
Combining this
with the irreducibility $\cM_{H,S}$ allows us to deduce
that $\cC'_{H,S}$ is irreducible.

Also we note that, for any $X \in \cM_{H,S}$, we have
\begin{align} \label{e:C'_S}
\dim \cC'_{H,S} &= \dim \cM_{H,S} + \dim \c_\p(X) \nonumber \\
&= \rank G - \ssrank H + \dim S + \dim(\p \cap \h) - j \nonumber \\
&= (\rank G - \ssrank H) + \dim(\p \cap \h) + (\dim S - j).
\end{align}

We define $\cC_{H,S} = P \cdot \cC'_{H,S} = \{(g \cdot X, g \cdot Y) \mid g \in P, (X,Y) \in \cC'_{H,S}\}$
to be the $P$-saturation of $\cC'_{H,S}$.
Then we have that $\cC_{H,S}$ is irreducible being the image
of the morphism $\phi : P \times \cC'_S \to \cC(\p)$ given by $\phi(g,(X,Y)) = (g \cdot X, g \cdot Y)$.
For $S' \in \cS_H$, we see that $\cC_{H,S} = \cC_{H,S'}$ if and only if $S$ and $S'$ lie in the same
$N_P(P \cap H)/(P \cap H)$-orbit.

We claim that
\begin{equation} \label{e:C_S}
\dim \cC_{H,S} = (\rank G - \ssrank H) + \dim \p + (\dim S - j).
\end{equation}
To prove this we consider the dimension of the fibres of the morphism $\phi : P \times \cC'_{H,S} \to \cC_{H,S}$
defined above.  We note that the dimension of these fibres is constant on $P$-orbits,
so it suffices to determine $\dim \phi^{-1}(X,Y)$ for $(X,Y) \in \cC'_{H,S}$.

Let $(X,Y) \in \cC'_{H,S}$.  We note that if $g \in P \cap H$, then
$(g,(g^{-1} \cdot X,g^{-1} \cdot Y)) \in \phi^{-1}(X,Y)$.  Now suppose that $(g,X',Y')) \in \phi^{-1}(X,Y)$,
so we have $X = g \cdot X'$ and $Y = g \cdot Y'$.
We have Jordan decompositions $X = X_s + X_n$ and $X' = X'_s + X'_n$, where $X_s, X'_s \in \z(\h)^\reg$,
because $X,X' \in \cM_S$.
Also we have $X_s = g \cdot X'_s$, and thus $\c_\p(X_s) = g \cdot \c_\p(X'_s)$.  Since,
$\c_\p(X_s) = \p \cap \h = \c_\p(X'_s)$, we have $g \in N_P(\p \cap \h) = N_P(P \cap H)$.
Hence, we have shown that
$$
P \cap H \sub \{g \in P \mid (g,(g^{-1} \cdot X,g^{-1} \cdot Y)) \in \phi^{-1}(X,Y)\} \sub N_P(P \cap H),
$$
which implies that $\dim \phi^{-1}(X,Y) = \dim(P \cap H)$.

We have seen that the dimension of each fibre of $\phi$ is equal to $\dim(P \cap H)$.
Hence, we have
$\dim \cC_{H,S} = \dim P + \dim \cC'_{H,S} - \dim(P \cap H)$ and substituting from \eqref{e:C'_S} gives
\eqref{e:C_S}.

Since $S$ is a sheet for the action of $P \cap H$ on
$\cN(\p \cap \h)$, we deduce that
\begin{equation} \label{e:C_S2}
\dim \cC_{H,S} \le (\rank G - \ssrank H) + \dim \p + \mod(P \cap H : \cN(\p \cap \h)).
\end{equation}

By construction we have the disjoint union
$$
\cC(\p) = \bigsqcup_{\substack{H \in \cH \\ S = \dot \cS_H}} \cC_{H,S},
$$
where $\dot \cS_H$ denotes a set of representatives for the
$N_P(P \cap H)/(P \cap H)$-orbits in $\cS_H$.
Moreover, the closure $\bar{\cC_{H,S}}$ of each $\cC_{H,S}$ is closed and irreducible.
Thus the irreducible
components of $\cC(\p)$ are given by some of the $\bar{\cC_{H,S}}$.

We have $\cC_{T,\{0\}}$ contains $\cC^\reg_\ss(\p)$ as an open subset, so that $\bar{\cC_{T,\{0\}}} = \bar{\cC^\reg_\ss(\p)}$
is an irreducible component by Lemma \ref{L:irrcomp}.

Now suppose that $\mod(P\cap H : \cN(\p \cap \h))< \ssrank H$ for
all $H \in \cH \setminus \{T\}$.  Then we have $\dim \bar{\cC_{H,S}} = \dim \cC_{H,S} < \dim \p + \rank G$
for all $S \in \cS_H$ and therefore $\bar{\cC_{H,S}}$ is not an irreducible component of $\cC(\p)$ by Lemma \ref{L:compdim}.
Therefore, $\bar{\cC^\reg_\ss(\p)}$ is the only irreducible component of $\cC(\p)$ and hence $\cC(\p)$
is irreducible.

Conversely, suppose that there is $H \in \cH \setminus \{T\}$ such that $\mod(P\cap H : \cN(\p\cap \h)) \ge \ssrank H$.
Then there is $S \in \cS_{H,j}$ for some $j \in \Z_{\ge 0}$ such that
$\dim S - j \ge \ssrank H$.  Then we have $\dim \bar{\cC_{H,S}} = \dim \cC_{H,S} \ge \rank G + \dim \p$.
Since $\cC_{H,S} \cap \cC^\reg_\ss(\p) = \varnothing$, and $\dim \cC_{H,S} \ge \dim \cC^\reg_\ss(\p)$,
we have $\bar{\cC_{H,S}} \nsubseteq \bar{\cC^\reg_\ss(\p)}$.  Hence, $\cC(\p)$ is not irreducible
by Lemma \ref{L:irrcomp}.

Finally, it is immediate from Lemma \ref{L:irrcomp} that $\dim \cC(\p) = \dim \p + \rank G$ when
$\cC(\p)$ is irreducible.
\end{proof}

We have the following immediate corollary, which gives a monotonicity
for irreducibility of $\cC(\p)$.

\begin{corollary} \label{C:levi}
Suppose that $\cC(\p)$ is irreducible
and let $H$ be a Levi subgroup of $G$ containing $T$.
Then $\cC(\p \cap \h)$ is irreducible.
\end{corollary}

In fact, the commuting variety of $\p$ gets much more complicated
as $\rank G - \ssrank L$ gets large, as explained, in particular for
the case $P = B$, in the next remark.

\begin{remark}
Let $H$ be a Levi subgroup of $G$ containing $T$
and let $\cX$ be an irreducible component of $\cC(\p \cap \h)$.
We remark that the methods in our proof of Theorem \ref{T:main} can be extended
to show that the closure of $P \cdot \cX = \{(g \cdot X,g \cdot Y) \mid
g \in P, (X,Y) \in \cX\}$ is an irreducible component of $\cC(\p)$.

Let $P = B$ be a Borel subgroup of $G$.  In this case $\cN(\b) = \u$,
and as a consequence of \cite[Theorem 3.1]{Ro1}, we have that $\mod(B:\u) \to \infty$ quadratically as $\ssrank G \to \infty$.
We may choose inside $G$ a sequence of Levi subgroup $1 = H_0 \le \dots \le H_{\ssrank G} = G$
such that $\ssrank H_j = j$ for each $j$.  Then we can see that for each $j$ such that $j \le \mod(B \cap H_j :\u \cap \h_j) <
\mod(B \cap H_{j+1} :\u \cap \h_{j+1})-1$, there are more irreducible components of $\cC(\b \cap \h_{j+1})$ than of $\cC(\b \cap \h_j)$.
Since $\mod(B:\u) \to \infty$ quadratically as $\ssrank G \to \infty$, we can deduce that the number of irreducible
components of $\cC(\b)$ gets arbitrarily large as $\ssrank G \to \infty$.  Indeed the number
of components grows at least linearly in the rank, and the dimension of some irreducible components
get much larger than $\dim \b + \rank G$.
\end{remark}

We end this section by recording a useful inductive result
regarding irreducibility of $\cC(\p)$.

\begin{proposition} \label{P:paracontain}
Let $P$ and $Q$ be parabolic subgroups of $G$.  Suppose that $P \le Q$.  Then $\cC(\q) = Q \cdot \cC(\p)$.
In particular, if $\cC(\p)$ is
irreducible, then $\cC(\q)$ is irreducible.
\end{proposition}

\begin{proof}
Let $(X,Y) \in \cC(\q)$ with Jordan decompositions $X = X_s+X_n$ and $Y = Y_s + Y_n$.
Then $X_n,Y_n \in \c_\g(X_s) \cap \c_\g(Y_s)$, which is the Lie
algebra of the Levi subgroup $H = C_G(X_s) \cap C_G(Y_s)$ of $G$.
By \cite[Corollary 5.5]{He}, there is a Springer isomorphism
$\phi : \cN(\h) \to \cU(H)$, where $\cU(H)$ denotes the unipotent
variety of $H$.  Since $\phi$ is $H$-equivariant we
deduce that $\phi(X_n),\phi(Y_n)$ commute.  Therefore,
$\phi(X_n)$ and $\phi(Y_n)$ lie in a Borel subgroup $B_H$ of $Q \cap H$.
Since Borel subgroups of $Q \cap H$ are conjugate, there exists $g \in Q \cap H$ such that $gB_Hg^{-1} \sub P \cap H$.
By $H$-equivariance, we have that $\phi$ sends $\cN(\h) \cap \p$ to $\cU(H) \cap P$ and thus
that $(g \cdot X, g \cdot Y) \in \cC(\p)$.  Hence, $\cC(\q) = Q \cdot \cC(\p)$.

Suppose $\cC(\p)$ that is irreducible. Then $\cC(\q)$ is the image of
the irreducible variety $Q \times \cC(\p)$ under the morphism $(g,(X,Y)) \mapsto (g \cdot X, g \cdot Y)$,
and thus is irreducible.
\end{proof}

We end this section with a discussion of when the commuting variety $\cC(\p)$ is equidimensional.
By Lemma \ref{L:irrcomp}, if $\cC(\p)$ is equidimensional then it must have dimension $\dim \p + \rank G$.
Now the proof of Theorem \ref{T:main} can be adapted to show that $\cC(\p)$ is equidimensional if and only if
$\mod(P\cap H : \cN(\p\cap \h)) \le \ssrank H$ for all Levi subgroups $H \ne T$ of $G$ containing $T$.

In the case $P = B$, it is now natural to consider whether $\cC(\b)$ is a complete intersection, as is
done in \cite[Section 6]{Ke}, and the argument in the last paragraph of the proof of \cite[Lemma 6.4]{Ke}
suffices to show that $\cC(\b)$ is a complete intersection whenever it is equidimensional.  The reason to consider
this question just for $P = B$ is that in this case the image of the commutator map $\b \times \b \to \b$ lies in
$\u$.  Whereas for general $P$ there is not a subspace of codimension $\rank G$ in $\p$ in which the image
of the commutator map $\p \times \p \to \p$ lies.

We remark that after this paper was completed, the paper \cite{Ba} of Basili appeared.  In {\em loc.\ cit.\ }
the question of whether $\cC(\b)$ is a complete intersection is considered in the case $G = \GL_n(k)$, and results
similar to those in this paper are obtained.

\section{On normality of \texorpdfstring{$\cC(\p)$}{C(p)}} \label{S:normality}

In this section we consider whether $\cC(\p)$ is a normal variety in the case where it is irreducible.
This question was considered by Keeton for the case $P = B$ in \cite[Section 6.2]{Ke}, and
we take a similar approach here.

We work with the commuting scheme $\cCS(\p)$, which is the subscheme of $\p \times \p$ defined by the
ideal $I(\p)$ of $k[\p \times \p]$ generated by the regular functions
$(X,Y) \mapsto f([X,Y])$ for $f \in \p^*$.   We let
$A(\p) = k[\p \times \p]/I(\p)$ be the algebra underlying $\cCS(\p)$.
The commuting scheme $\cCS(\l)$ of $\l$
is defined similarly, as are the ideal $I(\l)$ of $k[\l \times \l]$ and the algebra $A(\l) =
k[\l \times \l]/I(\l)$.

Our results depend on whether $A(\l)$ is Cohen--Macaulay,
so we include this as a hypothesis in the statements of our results.
Also for the case $P = B$, we have $\l = \t$ and $A(\t) = k[\t \times \t]$, which
is certainly Cohen--Maucaulay.  Further, we only consider cases where $\cC(\p)$ is
irreducible.

We start by showing that $\cCS(\p)$ is a Cohen--Macaulay  under these assumptions.

\begin{proposition} \label{P:CM}
Assume that $A(\l)$ is Cohen--Macaulay and that $\cC(\p)$ is irreducible.
Then $A(\p)$ is Cohen--Macaulay.
\end{proposition}

\begin{proof}
Let $J(\l)$ be the ideal of $A(\p)$ generated by $I(\l)$, and let
$B(\l) = k[\p \times \p]/J(\l)$.  We have $k[\p \times \p] \cong
k[\l \times \l] \otimes k[\u_P \times \u_P]$, and under this identification
$J(\l) = I(\l) \otimes k[\u_P \times \u_P]$.  Thus we have
that $B(\l) \cong (k[\l \times \l]/I(\l)) \otimes k[\u_P \times \u_P]$, which is Cohen--Macaulay by
\cite[Proposition 18.9]{Ei}.

For $\alpha \in \Phi(\u_P)$, we define $p_\alpha \in \u_P^*$ by $p_\alpha(\sum_{\beta \in \Phi(\u_P)} a_\beta E_\beta) = a_\alpha$
and $f_\alpha \in k[\p \times \p]$ by
$f_\alpha(X,Y) = p_\alpha([X,Y])$.
Let $J(\p)$ be the image of $I(\p)$ in $B(\l)$.
Since $I(\p)$ is generated by the set of $f_\alpha$ for $\alpha \in \Phi(\p)$,
we have that  $J(\p)$
is generated by the images of the $f_\alpha$ in $B(\l)$.
Further, we have that $A(\p) \cong B(\l)/J(\p)$.

The zero set of $J(\l)$ in $\p \times \p$ is $\cL$ as defined
in the proof of Lemma \ref{L:compdim} and has dimension
$\dim \l + \rank G + 2\dim \u_P$.  Our assumption that
$\cC(\p)$ is irreducible implies that it has dimension $\dim \p + \rank G = \dim \cL - \dim \u_P$.
Therefore, as $J(\p)$ is generated by $\dim \u_P = |\Phi(\u_P)|$ many elements and $B(\l)$ is Cohen--Macaulay,
we deduce that $A(\p) \cong B(\l)/J(\p)$ is Cohen--Macaulay, see \cite[Proposition 18.13]{Ei}.
\end{proof}

We remark that the only place that we use the fact that $\cC(\p)$ is irreducible in the proof above
is in saying that the dimension is $\dim \p + \rank G$.  Therefore, the proof is also valid
under the assumption that $\cC(\p)$ is equidimensional, because in this case it has dimension
$\dim \p + \rank G$ by Lemma \ref{L:irrcomp}.

In Theorem \ref{T:normality} we give a sufficient condition for $A(\p)$ to be normal.  Before we
move on to this it is helpful to make some observations about smooth points in $\cCS(\p)$,
which lead to Lemma \ref{L:regsmooth}.

Let $(X,Y) \in \cC(\p)$ considered as a closed point of $\cCS(\p)$.
The (scheme theoretic) tangent space of $\cCS(\p)$ at $(X,Y)$ is
$$
T_{(X,Y)}(\cCS(\p)) = \{(W,Z) \in \p \times \p \mid [X,Z]+[W,Y] = 0\}.
$$
Suppose that $\cC(\p)$ is irreducible so that $\dim \cC(\p) = \dim \p + \rank G$.
Then we have that $(X,Y)$ is a smooth point of $\cCS(\p)$ if and only if $\dim T_{(X,Y)}(\cCS(\p)) =
\dim \p + \rank G$, or equivalently $\dim([\p,X]+[\p,Y]) = \dim \p - \rank G$.
We recall that $X \in \p$ is regular if and only if $\dim \c_\p(X) = \rank G$, so that
$\dim([\p,X]) = \dim \p - \rank G$.  Therefore, if $X$ is regular, then
$\dim([\p,X]+[\p,Y]) = \dim \p - \rank G$.

In the previous paragraph, we have proved
the following lemma.

\begin{lemma} \label{L:regsmooth}
Assume that
that $\cC(\p)$ is irreducible.  Let $(X,Y) \in \cC(\p)$ considered as a closed point of $\cCS(\p)$ and
suppose that $X$ (or $Y$) is regular in $\p$.  Then $(X,Y)$ is a smooth point of $\cCS(\p)$.
\end{lemma}

We are now ready to proceed with Theorem \ref{T:normality}.
In the proof below we use the notation given in the proof of Theorem \ref{T:main}.

\begin{theorem} \label{T:normality}
Assume that $A(\l)$ is Cohen--Macaulay and that $\mod(P\cap H : \cN(\p\cap \h))< \ssrank H - 1$
for all Levi subgroups $H \ne T$ of $G$ containing $T$ with $\ssrank H > 1$.
Then $\cC(\p)$ is irreducible and normal.
\end{theorem}

\begin{proof}
By Theorem \ref{T:main}, we have that $\cC(\p)$ is irreducible, so
we just have to prove normality.  To do this we prove that $A(\p)$ is normal,
which in turn implies that $I(\p)$ is a prime ideal, and $k[\cC(\p)] = A(\p)$.

Since, $A(\p)$ is Cohen--Macaulay, it suffices by Serre's criterion for normality, see
\cite[Theorem 11.5]{Ei}, to show that the singular locus of
$\cCS(\p)$ has codimension at least 2.
We let $\cC(\p)^\ssm$ be the set of points in $\cC(\p)$,
which are smooth as closed points in $\cCS(\p)$.
Then we have  to show that $\cC(\p) \setminus \cC(\p)^\ssm$
has codimension at least 2.

Let $(X,Y) \in \cC_{T,\{0\}}$.
Then $X$ is regular (semisimple), so that $(X,Y)$ is a
smooth point of $\cCS(\p)$ by Lemma \ref{L:regsmooth}.
Thus $\cC_{T,\{0\}} \sub \cC(\p)^\ssm$.

Consider a Levi subgroup $H$ of $G$ containing $T$ with $\ssrank H = 1$.
We have two possibilities for $H \cap P$, either this is equal to $H$ or is a Borel subgroup of
$H$; it turns out that the analysis of these two cases can be done simultaneously.

There are two $P \cap H$-orbits in $\cN(\p \cap \h)$ namely
the regular nilpotent orbit and the zero orbit.  Thus $\cS_H$ has these two orbits as its elements.

First consider $S = \{0\} \in \cS_H$, and look at $\cC'_{H,S}$.
We have that $\cC'_{H,S} = \{(X,Y) \mid X \in \z(\h)^\reg, Y \in \h\}$.
Let $\cY = \{(X,Y) \mid X \in \z(\h)^\reg, Y \in \h^\reg\} \sub \cC'_S$; we recall
that $\z(\h)^\reg = \{X \in \z(\h) \mid \c_\g(X) = \h\}$ and that
$\h^\reg$ denotes the set of regular elements in $\h$.  By Lemma \ref{L:regsmooth}
we have that $\cY$ and, thus $P \cdot \cY \sub \cC(\p)^\ssm$.
Also $P \cdot \cY$ is the subset of $\cC_{H,S}$ of those $(X,Y)$ for which
$Y$ is regular, and is thus open in $\cC_{H,S}$.
Therefore, the complement of $P \cdot \cY$ has codimension at least 1 in $\cC_{H,S}$.
In the proof of Theorem \ref{T:main}, we have shown that $\cC_{H,S}$ has codimension 1
in $\cC(\p)$.  So we can deduce that $\cC_{H,S} \setminus (P \cdot \cY)$ has codimension
at least 2 in $\cC(\p)$.

Next consider the case $S$ is the regular orbit of $P \cap H$ in $\cN(\p \cap \h)$.
Here we observe that any element of $\cC'_{H,S}$ is of the form
$(X+Z,Y+aZ)$, where $X,Y \in \z(\h)^\reg$, $Z \in \cN(\p \cap \h)$ is regular in $\h$, and $a \in k$.
Then $X+Z$ is regular in $\g$, so it follows from Lemma \ref{L:regsmooth}
that $\cC'_{H,S} \sub \cC(\p)^\ssm$, and thus
$\cC_{H,S} = P \cdot \cC'_{H,S} \sub \cC(\p)^\ssm$.

Now consider the decomposition
$$
\cC(\p) = \bigsqcup_{\substack{H \in \cH \\ S = \dot \cS_H}} \cC_{H,S},
$$
given in the proof of Theorem \ref{T:main}.

Let $H \in \cH$ with $\ssrank H > 1$.
Then by assumption we have that $\mod(P\cap H : \cN(\p\cap \h))< \ssrank H - 1$, so
$\dim \cC_{H,S} < \dim \p + \rank -1 $ by \eqref{e:C_S2}.  Thus $\cC_{H,S}$ has codimension at
least 2 in $\cC(\p)$ for each $S \in \cS_H$.

Moreover, we have shown above that $(\cC(\p) \setminus \cC(\p)^\ssm) \cap \cC_{H,S}$
has codimension at least 2 in $\cC(\p)$, for $H$ of semisimple rank 0 or 1 and $S \in \cS_H$.

Hence, we have proved that $(\cC(\p) \setminus \cC(\p)^\ssm)$ has codimension at
least 2 in $\cC(\p)$ as required.
\end{proof}

In the case $P = B$ is a Borel subgroup, we can remove the assumption
that $A(\l)$ is Cohen--Macaulay in Theorem \ref{T:normality},
as $\l = \t$ and $A(\t) = k[\t \times \t]$.  The converse
also holds as stated Theorem \ref{T:borelnormality} below.
In its proof we use the setup given in the proof
of Theorem \ref{T:main}.

\begin{theorem} \label{T:borelnormality}
Let $B$ be a Borel subgroup of $G$ and $\u = \cN(\b)$ the Lie algebra
of the unipotent radical of $B$.
Then $\cC(\b)$ is irreducible and normal if and only if $\mod(B\cap H : \u \cap \h) < \ssrank H - 1$ for
all Levi subgroups $H$ of $G$ containing $T$ with $\ssrank H > 1$.
\end{theorem}

\begin{proof}
For the case $P = B$, we have $L = T$ and $A(\t) = k[\t \times \t]$,
which is certainly Cohen--Macaulay.
Thus if $\mod(B\cap H : \u \cap \h) < \ssrank H - 1$ for all $H$, then $\cC(\b)$ is normal by
Theorem \ref{T:normality}.  So we just have to prove the converse.

First we make the
simplifying observation that, under the assumption
that $\cC(\b)$ is irreducible, we have that $A(\b)$ is reduced.
To prove this we consider the functions $f_\alpha \in k[\b \times \b]$
for $\alpha \in \Phi(\u) = \Phi^+$, as defined in the proof of Proposition
\ref{P:CM}.

Let $(X,Y) \in \cC(\b)$.
The differential $(df_\alpha)_{(X,Y)} : T_(X,Y)(\b \times \b) \to k$ can be viewed as
a linear map $\b \times \b \to k$ and a calculation shows that this is the map
$(W,Z) \mapsto f_\alpha([X,Z] + [W,Y])$.  Thus the differentials $(df_\alpha)_{(X,Y)}$ for $\alpha \in \Phi^+$ are
linearly independent whenever $(X,Y) \in \cC(\b)$ is such that $[\b,X] + [\b,Y] = \u$;
this holds for example when $X$ is regular semisimple.  Now we can deduce that $I_\b$ is radical and thus $A(\b)$ is reduced,
see for example
\cite[Lemma 7.1]{Ja}.

We just have to consider the
case $\mod(B\cap H : \u \cap \h) = \ssrank H - 1$ for some Levi subgroup $H$ of $G$ with $\ssrank H > 1$.
In this case we prove that the singular locus of $\cC(\b)$ has codimension equal to 1;
we can work just with the variety here, because $A(\b)$ is reduced.
Then we can apply Serre's criterion for normality to deduce that $\cC(\b)$ is not normal.

We have that $H$ is conjugate in $G$ to a standard Levi subgroup, so since $B \cap H$
is a Borel subgroup of $H$, we may as well assume that $H$ is a standard Levi subgroup.
Let $J \sub \Pi$ be such that $\Phi(\h) = \Phi_J$.

In the notation given in the proof of Theorem \ref{T:main}, we have that there is
some $S \in \cS_H$ such that $\dim \cC_{H,S} = \dim \b + \rank G - 1$.
From the proof of \cite[Theorem 5.1]{GMR} we can assume that
$Z(H)C_U(U \cap H)$ has finite index in $C_{B \cap H}(X)$ for any $X \in S$;
in other words the connected component of the centralizer of $X$ in the $B \cap \cD H$
is unipotent, where $\cD H$ denotes the derived subgroup of $H$.  This implies
that $\mod(U \cap H : \u \cap \h) = 2\ssrank H - 1$.
Let $j \in \Z_{\ge 0}$, be such
that $S \in (\cS_H)_j$.  Then we have $\dim S = j + \ssrank H  - 1$.

Let $X \in S$.  Then we have
$\dim C_{B \cap H}(X) = \dim (B \cap H) - j$ for any $X \in S$, so
$\dim C_{U \cap H}(X) = \dim (U \cap H) - j - \ssrank H$.
By Theorem \ref{T:Pseparableintro}, we have $\dim C_{B \cap H}(X) = \dim \c_{\b \cap \h}(X)$, and
also we have $\dim C_{U \cap H}(X) = \dim \c_{\u \cap \h}(X)$,
see for example \cite[Corollary 4.3]{Go}.
From this we deduce that $\c_{\b \cap \h}(X) = \z(\h) \oplus \c_{\u \cap \h}(X)$.

Consider the set $S^\reg$ of elements of $\u \cap \h$ of the form
$$
\sum_{\alpha \in \Phi_J \cap \Phi^+} a_\alpha E_\alpha,
$$
where $a_\alpha \ne 0$ for all $\alpha \in J$.  The elements of $S^\reg$ are precisely the
regular nilpotent elements in $\u \cap \h$, and they form a single $B \cap H$-orbit.
Moreover, $S^\reg$ is a sheet for the action of $B \cap H$ on $\u \cap \h$ and
$S^\reg \in (\cS_H)_{\dim(\u\cap \h)}$.  We have $\dim S^\reg = \dim(\u\cap \h) < \ssrank H + \dim(\u\cap \h) -1$,
because $\ssrank H > 1$.  Therefore, $S \ne S^\reg$.

For $\beta \in J$, we let $(\u \cap \h)^\beta = \bigoplus_{\alpha \in \Phi(\u \cap \h) \setminus \{\beta\}} \g_\alpha$.
Then we have $(\u \cap \h) \setminus S^\reg = \bigcup_{\beta \in J} (\u \cap \h)^\beta$, and it follows
that $S \sub (\u \cap \h)^\beta$ for some $\beta \in J$.

Next we consider $\cC(\u \cap \h)$, which has dimension $\dim (\u \cap \h) + \mod(U \cap H : \u \cap \h)$
by Lemma \ref{L:dimcg}.
We define $\cC(\u \cap \h)_S = \cC(\u \cap \h) \cap (S \times (\u \cap \h))$.
The projection onto the first component $\cC(\u \cap \h)_S \to S$ is surjective and the fibre of
$X \in S$ is equal to $\{X\} \times \c_{\u \cap \h}(X)$.  As seen above we have that
$\c_{\u \cap \h}(X) = \dim (U \cap H) - j - \ssrank H$ and that this is independent of $X \in S$.
Therefore,
\begin{align*}
\dim \cC(\u \cap \h)_S  &= \dim S + \dim(U \cap H) - j + \ssrank H \\
&= j + \ssrank H - 1 + \dim(U \cap H) - j + \ssrank H \\
&= \dim(U \cap H) + 2\ssrank H - 1\\
&= \dim(\u \cap \h) + \mod(U \cap H : \u \cap \h) \\
&= \dim \cC(\u \cap \h).
\end{align*}
Further, $\cC(\u \cap \h)_S$ is irreducible, so its closure in $\cC(\u \cap \h)$
is an irreducible component of $\cC(\u \cap \h)$.

Therefore, $\bar{\cC(\u \cap \h)_S}$ is stable under the flip $(X,Y) \mapsto (Y,X)$,
see for example \cite[Lemma 2.3]{GR}.  Since $S \sub (\u \cap \h)^\beta$ we deduce that
$\cC(\u \cap \h)_S \sub (\u \cap \h)^\beta \times (\u \cap \h)^\beta$, so
that $\c_{\u \cap \h}(X) \sub (\u \cap \h)^\beta$ for any $X \in S$.

Now consider $\cC'_{H,S}$ as defined in the proof of Theorem \ref{T:main}.
Let $(X,Y) \in \cC'_{H,S}$.  Then $X = X_1 + X_2$, where $X_1 \in \z(\h)^\reg$
and $X_2 \in S$.
We have shown above that $\c_{\b \cap \h}(X_2) = \z(\h) \oplus \c_{\u \cap \h}(X_2)$,
so we can write $Y = Y_1 + Y_2$, where $Y_1 \in \z(\h)$, $Y_2 \in \c_{\u \cap \h}(X_2)
 \sub (\u \cap \h)^\beta$.

Now we aim to show that $(X,Y)$ is a singular point of $\cC(\b)$.
We recall from the discussion before Theorem \ref{T:normality} that $(X,Y)$
is a smooth point of $\cC(\b)$ if and only if $\dim([\b,X]+[\b,Y]) = \dim \u$.  However,
we know that $X \in \z(\h) \oplus (\u \cap \h)^\beta$, which implies that
$[\b,X] \sub \u^\beta$.  Similarly, $Y \in \z(h) \oplus (\u \cap \h)^\beta$,
so that $[\b,Y] \sub (\u \cap \h)^\beta$.  Thus, $[\b,X]+[\b,Y] \sub \u^\beta$
so $\dim([\b,X]+[\b,Y]) < \dim \u$.  Hence, $(X,Y)$ is a singular point of $\cC(\b)$.

We deduce that all elements of $\cC_{H,S} = B \cdot \cC'_{H,S}$ are singular in
$\cC(\b)$.  Further, $\dim \cC_{H,S} = \dim \cC(\b) - 1$.  Hence, $\cC(\b)$ is not normal
by Serre's criterion for normality.
\end{proof}

\section{Classification of irreducibility and normality of \texorpdfstring{$\cC(\b)$}{C(b)}}

We consider the case where $P = B$ is a Borel subgroup, where
we can give a full classification of irreducibility and normality of $\cC(\b)$.
This requires recent results in \cite{GMR} and \cite{PS} giving $\mod(B : \u)$ for simple $G$ of
sufficiently large rank, here we recall that $\cN(\b) = \u$ is the nilradical of $\b$.
We also use lower bounds for $\mod(B : \u)$
established in \cite[Theorem 3.1]{Ro1}.

In \cite{GMR} a parametrization of the coadjoint orbits of $U$ in $\u^*$
is given for $G$ of rank at most $8$ apart from $G$ of type $\rE_8$.  It
is known that $\mod(U:\u) = \mod(U:\u^*)$, see \cite[Theorem 1.4]{Ro2}.  Also in \cite[Theorem 5.1]{GMR} it
is proved that $\mod(U:\u) = \mod(B:\u) + \ssrank G$.  Thus
values of $\mod(B:\u)$ for $G$ up to rank $8$ apart from $G$ of type $\rE_8$
were determined.  This extended previously known values of $\mod(B:\u)$
given in \cite[Tables II and III]{JR}.

The results in \cite{PS} can be used to determine $\mod(B:\u)$ for
$G = \GL_n(k)$ and $n \le 16$, as we explain below.   Let $q$ be a prime power,
let $U(q)$ be the subgroup of upper unitriangular matrices in $\GL_n(q)$
and let $\u^*(q)$ be the dual space of the space $\u(q)$ of strictly upper
triangular matrices in $\gl_n(q)$.  Then $U(q)$ acts on $\u^*(q)$ via the coadjoint
action.  The number $k(U(q),\u^*(q))$ of coadjoint orbits of $U(q)$ in $\u^*(q)$ is determined
for $n \le 16$ in \cite{PS} and this number is shown to be a polynomial in $q$,
see \cite[Theorem 1.2]{PS}.  Although \cite{PS} only deals with finite fields,
the methods used can be adapted to apply for
other fields.  In particular, this means that the calculations
carried out as part of \cite{PS} can be used to determine
$\mod(U:\u^*)$ for $G = \GL_n(k)$ and $n \le 16$; moreover, we
see that $\mod(U:\u^*)$ is equal to the degree of the polynomial in $q$ giving
$k(U(q),\u^*(q))$.  Combining this with the fact that $\mod(U : \u^*) = \mod(U : \u) = \mod(B:\u) + \ssrank G$,
we deduce the values of $\mod(B:\u)$.

Combining the results in \cite{GMR}, \cite{PS} and \cite{Ro1}, gives
Tables 1--5, containing the exact value or a lower bound for $\mod(B:\u)$ for $G$ of low rank.  We note that
in higher rank cases the lower bounds from \cite{Ro1} do give that $\mod(B:\u) > \rank G$ in these cases.

\begin{center}
\begin{table}[htb]
\renewcommand{\arraystretch}{1.1}
\begin{tabular}{c|ccc ccc ccc ccc ccc cc}
Type of $G$ & $\rA_1$ & $\rA_2$ & $\rA_3$ & $\rA_4$ & $\rA_5$ & $\rA_6$ & $\rA_7$ & $\rA_8$ & $\rA_9$  \\ \hline
$\mod(B:\u)$ & 0 & 0 & 0 & 0 & 1 & 1 & 2 & 3 & 4  \\  \hline

Type of $G$ & $\rA_{10}$ & $\rA_{11}$ & $\rA_{12}$ &  $\rA_{13}$ & $\rA_{14}$ & $\rA_{15}$ & $\rA_{16}$ & $\rA_{17}$ \\ \hline
$\mod(B:\u)$ & 5 & 7 & 8 & 10 & 12 & 14 & $\ge 16$ & $\ge 19$
\end{tabular}
\medskip
\caption{Modality of the action of $B$ on $\u$ for $G$ of type $\rA$}
\end{table}
\end{center}

\begin{center}
\begin{table}[htb]
\renewcommand{\arraystretch}{1.1}
\begin{tabular}{c|cc ccc cc}
Type of $G$ & $\rB_2$ & $\rB_3$ & $\rB_4$ & $\rB_5$ & $\rB_6$ & $\rB_7$ & $\rB_8$ \\ \hline
$\mod(B:\u)$ & 0 & 1 & 2 & 3 & 5 & 7 & 9
\end{tabular}
\medskip
\caption{Modality of the action of $B$ on $\u$ for $G$ of type $\rB$}
\end{table}
\end{center}

\begin{center}
\begin{table}[htb]
\renewcommand{\arraystretch}{1.1}
\begin{tabular}{c|cc ccc c}
Type of $G$ & $\rC_3$ & $\rC_4$ & $\rC_5$ & $\rC_6$ & $\rC_7$ & $\rC_8$ \\ \hline
$\mod(B:\u)$ & 1 & 2 & 3 & 5 & 7 & 9
\end{tabular}
\medskip
\caption{Modality of the action of $B$ on $\u$ for $G$ of type $\rC$}
\end{table}
\end{center}

\begin{center}
\begin{table}[htb]
\renewcommand{\arraystretch}{1.1}
\begin{tabular}{c|cc ccc }
Type of $G$ & $\rD_4$ & $\rD_5$ & $\rD_6$ & $\rD_7$ & $\rD_8$ \\ \hline
$\mod(B:\u)$  & 1 & 2 & 4 & 5 &  8
\end{tabular}
\medskip
\caption{Modality of the action of $B$ on $\u$ for $G$ of type $\rD$}
\end{table}
\end{center}

\begin{center}
\begin{table}[h!tb]
\renewcommand{\arraystretch}{1.1}
\begin{tabular}{c|cc ccc }
Type of $G$  & $\rG_2$ & $\rF_4$ & $\rE_6$ & $\rE_7$ & $\rE_8$ \\ \hline
$\mod(B:\u)$  & 1 & 4 & 5 & 10 & $\ge 20$
\end{tabular}
\medskip
\caption{Modality of the action of $B$ on $\u$ for $G$ of exceptional type}
\end{table}
\end{center}

From  Tables 1--5 and Theorem \ref{T:main}, we immediately deduce Theorem \ref{T:borelclass}.
Also from Tables 1--5 and Theorem \ref{T:borelnormality}, we can deduce the classification of
when $\cC(\b)$ is normal given in the following theorem.

\begin{theorem} \label{T:borelnormalityclass}
Let $B$ be a Borel subgroup of $G$. Then
$\cC(\b)$ is irreducible and normal if and only if the type of each simple component of $G$ is one of the following.
\begin{itemize}
\item $\rA_l$ for $l \le 14$;
\item $\rB_l$ for $l \le 5$;
\item $\rC_l$ for $l \le 5$; or
\item $\rD_l$ for $l \le 7$.
\end{itemize}
\end{theorem}

We end this paper by briefly discussing some cases where we can determine whether
$\cC(\p)$ is irreducible or reducible for $P \ne B$.  In general little is known
about the values of $\mod(P:\cN(\p))$, so that we can only present limited results in this direction.

First we note that by Proposition \ref{P:paracontain} and Theorem \ref{T:borelclass}, we have
$\cC(\p)$ is irreducible whenever $G$ satisfies the conditions
in Theorem \ref{T:borelclass}.

In \cite[Table 1]{Ro1} R\"ohrle lower bounds for $\mod(P:\u_P)$ are given, where $\u_P$ is
the Lie algebra of the unipotent radical of $P$.  Of course, we have
$\mod(P: \u_P) \le \mod(P:\cN(\p))$.  From these lower bounds, we can determine
many instances where $\cC(\p)$ is reducible.  In particular, these lower bounds
are quadratic in $\ssrank G - \ssrank L$, so for a fixed value of $\ssrank L$, we have that
$\cC(\p)$ is reducible if $\rank G$ is sufficiently large.

If the number of $P$-orbits in $\cN(\p)$ is finite, so that $\mod(P:\cN(\p)) = 0$,
then certainly $\cC(\p)$ is irreducible.
It follows from results of Murray in \cite{Mur} that $\mod(P:\cN(\p)) = 0$ for
$P$ a maximal parabolic subgroup of $\GL_n(k)$ such that one block
has size less than 6.
Very recent work of Bulois--Boos in \cite[Main theorem]{BB} gives a classification of
cases where $\mod(P:\cN(\p)) = 0$ for $G = \GL_n(k)$.  Further,
in \cite[\S 6]{BB}, there is a discussion of cases of higher modality. In particular,
it is shown that for the maximal parabolic subgroup $P$ of $\GL_n(k)$ with block sizes
$(200,400)$, we have that $\cC(\p)$ is reducible.

\end{document}